\documentclass[a4paper]{amsart}
\usepackage[utf8]{inputenc}
\usepackage{amssymb}
\usepackage{mathtools}
\usepackage{microtype}
\usepackage{amsthm}
\usepackage{thmtools}
\usepackage[capitalize,nosort]{cleveref}
\usepackage[final,notref,notcite]{showkeys}
\usepackage{enumitem}
\usepackage[mathlines]{lineno}
\usepackage[draft,footnote=true,margin=false,marginclue]{fixme}
\usepackage[backend=biber,style=alphabetic,bibstyle=alphabetic]{biblatex}
\declaretheorem{theorem}
\declaretheorem[sharenumber=theorem]{corollary}
\declaretheorem[sharenumber=theorem]{lemma}
\declaretheorem[sharenumber=theorem]{proposition}
\declaretheorem[sharenumber=theorem,style=remark]{remark}
\declaretheorem[numbered=no,style=remark]{problem}
\declaretheorem[sharenumber=theorem,style=definition]{definition}
\declaretheorem[sharenumber=theorem,style=definition]{example}

\DeclarePairedDelimiterX\ind[2]{\lvert}{\rvert}{#1:#2}

\newcommand{\cent}[2]{\operatorname{C}_{#1}\lparen #2\rparen}
\newcommand{\covdeg}[1]{\operatorname{cd}\lparen #1\rparen}
\newcommand{\cond}[2]{#2_{#1}}
\newcommand{\fieldext}[2]{\mathcal{F}(#1/#2)}
\DeclareMathOperator{\K}{\mathcal{K}}

\title[Ideally \(r\)-constrained subalgebras]{Ideally \(r\)-constrained graded Lie subalgebras \\
of maximal class algebras}

\subjclass{Primary 17B70; secondary 17B65, 17B50}

\keywords{Ideally \(r\)-constrained Lie algebras, Lie algebras of maximal class, just-infinite dimensional Lie algebras, thin algebras, graded Lie algebras}

\author[M.~Avitabile]{Marina~Avitabile} \email{marina.avitabile@unimib.it}
\address{Dipartimento di Matematica e Applicazioni\\
 Università degli Studi di Milano - Bicocca\\
via Cozzi, 55\\
 I-20125 Milano\\
 Italy}

\author[N.~Gavioli]{Norberto~Gavioli} \email{norberto.gavioli@univaq.it}
 \address{Dipartimento di Ingegneria e Scienze dell'Informazione e Matematica\\
 Università degli Studi dell'Aquila\\
 via Vetoio\\
 I-67100 Coppito (AQ)\\
 Italy}

\author[V.~Monti]{Valerio~Monti} \email{valerio.monti@uninsubria.it}
 \address{Dipartimento di Scienza e Alta Tecnologia\\
Università degli Studi dell'Insubria\\
 Via Valleggio, 11\\
I-22100 Como\\
Italy}
\thanks{The first and third author thank the Dipartimento di Ingegneria e Scienze dell'Informazione e Matematica, Università dell'Aquila, for the kind hospitality.}
\thanks{The authors are members of INdAM-GNSAGA}

\begin{document}

\begin{abstract}
Let \(E\supseteq F\) be a field extension and \(M\) a graded Lie 
algebra of maximal class over \(E\). We investigate the \(F\)-subalgebras 
\(L\) of \(M\), generated by elements of  degree~\(1\). We provide conditions 
for \(L\) being either ideally \(r\)-constrained or not just infinite. We 
show by an example that those conditions are tight. Furthermore, we determine 
the structure of \(L\) when the field extension \(E\supseteq F\) is finite. 

A class of ideally \(r\)-constrained  Lie algebras which are not 
\((r-1)\)-constrained is explicitly constructed, for  every \(r\geq 1\).
\end{abstract}

\maketitle

\section{Introduction}

Ideally \(r\)-constrained Lie algebras arise naturally as generalisations of 
thin Lie algebras. The latters are positively graded, infinite dimensional  
Lie algebras, generated in degree \(1\) and such that every non-zero graded 
ideal is trapped between two consecutive Lie powers of the Lie algebra. As an 
immediate consequence of the definition, every homogeneous component of a 
thin algebra has dimension \(1\) or \(2\). 

Thin algebras have been introduced in~\cite{CMNS1996} and are widely studied.  
Although classification results have been provided for several classes of 
these algebras, a special class of thin algebras, that is those algebras 
whose homogeneous components, except the second, have dimension \(2\), is 
less understood. 

In~\cite{GMY2001} is proved that all metabelian thin Lie algebras belong to 
this class and they are in one-to-one correspondence with the quadratic 
extensions of the underlying field \(F\). These results have been extended to 
the non-metabelian case in~\cite{ACG+2023}. The main idea of that paper is to 
consider a quadratic field extension \(E\supseteq F\) and a Lie algebra \(M\) 
of maximal class over \(E\), that is \(M\) is a thin algebra over \(E\) whose 
homogeneous components, except the first, have dimension \(1\) (see 
\cref{sec:pre} for details). Thin algebras where all homogeneous components, 
but the second one, have dimension \(2\), are realised as \(F\)-subalgebras 
of \(M\), generated by two elements of degree \(1\). 

In this paper we focus on ideally \(r\)-constrained algebras, with \(r>0\) an 
integer. According to~\cite{GM2002}, we say that a Lie algebra 
\(L\coloneqq\bigoplus_{i\ge 1}L_{i}\) is  \emph{ideally \(r\)-constrained} 
(or $r$-constrained for short) if, for every non-zero graded ideal \(I\) of 
\(L\), there exists an integer \(i\) such that \(L^{i}\supseteq I\supseteq 
L^{i+r}\), where \(L^k=\bigoplus_{j\ge k} L_j\) is the \(k\)-th Lie power of 
\(L\). When \(r=1\) we will simply say that \(L\) is \emph{ideally 
constrained}. Thus, an ideally constrained Lie algebra whose first 
homogeneous component has dimension \(2\) is thin. Note also that a finitely 
generated \(r\)-constrained Lie algebra is {\em just infinite}, that is every 
non-zero graded ideal of the algebra has finite codimension. Periodic just 
infinite Lie algebras have been studied in \cite{GMS2004}, where it is proved 
that those algebras are necessarily \(r\)-constrained. 

In the spirit of~\cite{ACG+2023}, we consider an arbitrary  field extension 
\(E\supseteq F\) and a Lie algebra \(M\) of maximal class over \(E\). We 
provide conditions for an \(F\)-subalgebra \(L\) of \(M\), generated in 
degree \(1\), to satisfy a dichotomy, that is to be either ideally 
\(r\)-constrained or not just infinite. Precisely, we associate to \(L\)  a 
sequence of intermediate fields \(F_{i}\)'s, depending on the intersection of 
the first homogeneous component of \(L\) with the \(2\)-{\em step 
centralisers} \(C_{i}\)'s of \(M\) (see \cref{sec:pre} for the definition of 
the \(2\)-step centralisers of a Lie algebra of maximal class). Differently 
from~\cite{ACG+2023},  where only the dimensions of the fields \(F_i\)'s 
matter, in this paper the field \(\K\) generated by the \(F_{i}\)'s plays a 
role in the above mentioned dichotomy. In particular, we can determine the 
structure of \(L\) when the field extension \(E\supseteq F\) is finite or 
when \(E\) is algebraic over \(F\) and \(M\) has only finitely many 
\(C_{i}\)'s. 

The paper is structured as follows. In \cref{sec:pre}, after recalling 
standard definitions and properties of Lie algebras of maximal class, we 
introduce the techniques needed to define the fields \(F_{i}\)'s. In 
\cref{sec:main}, we collect the properties of the \(F_{i}\)'s and we prove 
our main result in \cref{thm:dichotomy}. We end the paper with examples and 
open problems in \cref{sec:examples}. 

In particular,  for every \(r\geq1\) we exhibit an 
\(r\)-constrained Lie algebra which is not \((r-1)\)-constrained. 
We also show that the 
assumptions of \cref{thm:dichotomy} are tight, by considering the case \(E\) 
is transcendental over \(F\). 

\section{Preliminaries}\label{sec:pre}
If \(X\) is a subset of a (left) module \(V\) over a ring \(R\), we denote by 
\(RX\) the \(R\)-submodule generated by \(X\). We will simply write \(Rx\) 
when \(X=\{x\}\). By a Lie algebra we mean a graded Lie algebra 
\(L\coloneqq\bigoplus_{i\ge 1}L_{i}\) over some field, generated by its first 
homogeneous component \(L_{1}\). If not otherwise stated, we shall assume 
that \(L_i\ne 0\) for every \(i\): in particular \(L\) is infinite 
dimensional. Given a Lie algebra \(L\) and subsets \(X\) and \(Y\) of \(L\), 
we denote by \([X,Y]\) the additive subgroup of \(L\) generated by the 
elements \([x,y]\) as \(x\) ranges in \(X\) and \(y\) ranges in \(Y\). We 
will simply write \([X,y]\) when \(Y=\{y\}\).  In particular, 
\(L_{i}=[L_{i-1},L_{1}]\) for every \(i\ge 2\), as we are assuming that \(L\) 
is generated by \(L_1\). A Lie algebra \(M\coloneqq\bigoplus_{i\ge 1}M_{i}\) 
over a field \(E\) is said to be of \emph{maximal class} if 
\(\dim_{E}M_{1}=2\) and \(\dim_{E}M_{i}=1\) for every \(i\ge2\). The 
\emph{\(2\)-step centraliser} \(C_{i}\) for \(i\ge 2\) is 
\(\cent{M_{1}}{M_{i}}\): this is a subspace of dimension \(1\). Note that for 
every non-zero \(l_{i}\in M_{i}\) we have \(C_{i}=\cent{M_{1}}{l_{i}}\). The 
following lemma, which we quote for easy reference, is implicitly stated in 
\cite[Proposition~4.1]{CN2000}. 
\begin{lemma}\label{lem:specialisation}
  Let \(M\) be a Lie algebra of maximal class and \(C\) be a \(2\)-step 
  centraliser. If \(t\) is the smallest integer such that \(C_{t}=C\), then 
  in every interval of integers of length \(t\) there is at least one \(j\) 
  such that \(C_{j}=C\). In particular, there are infinitely many occurrences 
  of \(C\). 
\end{lemma}
We say that a Lie algebra \(L\coloneqq\bigoplus_{i\ge 1}L_{i}\) is 
\emph{ideally \(r\)-constrained} if, for every non-zero graded ideal \(I\), 
there exists an integer \(i\) such that \(L^{i}\supseteq I\supseteq 
L^{i+r}\), where \(L^k=\bigoplus_{j\ge k} L_j\) is the \(k\)-th Lie power of 
\(L\). Equivalently, \(L\) is ideally \(r\)-constrained if, for every 
positive integer \(i\) and every non-zero homogeneous element \(z\) of degree 
\(i\), we have that \([z,\prescript{}{r}L_{1}]=L_{i+r}\), where 
\([z,\prescript{}{k}L_{1}]\) is defined recursively by setting 
\([z,\prescript{}{1}L_{1}]\coloneqq[z,L_{1}]\) and 
\([z,\prescript{}{k}L_{1}]\coloneqq[[z,\prescript{}{k-1}L_{1}].L_{1}]\) for 
\(k\ge2\). When \(r=1\) we will simply say that \(L\) is \emph{ideally 
constrained}. A Lie algebra of maximal class is ideally constrained. A 
\emph{just infinite} (dimensional) Lie algebra is an algebra whose non-zero 
graded ideals have finite codimension. A finitely generated ideally 
\(r\)-constrained Lie algebra is just infinite. 

\begin{definition}\label{def:E_U}
  Let \(V\) be a vector space over a field \(E\) and \(U\) a subgroup of 
  the additive group of \(V\). We denote by \(\cond{U}{E}\) the subset 
  of \(E\) of the elements \(\alpha\) such that \(\alpha U\subseteq U\).
\end{definition}

\begin{lemma}\label{lem:modulering}
  With the previous notation, the subset 
  \(\cond{U}{E}\) of \(E\) is a subring of \(E\). The subset \(U\) is an 
  \(R\)-module with respect to a given subring \(R\) of \(E\) if and only if 
  \(R\subseteq \cond{U}{E}\): in particular, \(U\) is an 
  \(\cond{U}{E}\)-module. If there exists a subfield \(F\) of \(\cond{U}{E}\) 
  such that \(\dim_{F}U\) is finite, then \(\cond{U}{E}\) is a subfield of 
  \(E\). 
\end{lemma}
\begin{proof}
  We prove the last statement since the rest is trivial. If \(U=0\) then 
  \(\cond{U}{E}=E\); if \(U\ne0\), take \(u\in U\setminus \{0\}\) and note 
  that the map from \(\cond{U}{E}\) to \(U\) sending \(k\) in \(ku\) is an 
  injective \(F\)-linear map so \(\dim_{F}\cond{U}{E}\) is finite, and 
  therefore \(E_U\) is a field. 
\end{proof}

\begin{remark}\label{rem:increasingring}
  Let \(L\) be a Lie algebra over some field \(E\). If \(U\) is an additive 
  subgroup of \(L\) and \(X\) is a subset of \(L\), then 
  \(\cond{[U,X]}{E}\supseteq \cond{U}{E}\). 
\end{remark}

\begin{lemma}\label{lem:cosetgeneratesamefield}
  Let \(X\) and \(Y\) be subsets of a field \(E\) such that \(X=Yu\) for some 
  non-zero \(u\) in \(E\) and \(1\in X\cap Y\). Then the subfield generated by 
  \(X\) and the subfield generated by \(Y\) coincide. 
\end{lemma}
\begin{proof}
  Since \(1\in Y\), we have that \(u=1u\in X\). Thus, given \(y\) in \(Y\), 
  we get that \(y=yu\cdot u^{-1}\) belongs to the subfield generated by 
  \(X\). For the other inclusion, just note that \(Y=Xu^{-1}\). 
\end{proof}

\begin{definition}\label{def:linearalgebra}
 Let \(E\supseteq F\) be fields, \(V\) an \(E\)-vector space, \(C\) an 
  \(E\)-subspace of \(V\) of codimension \(1\), and \(W\) an \(F\)-subspace 
  of \(V\), such that \(W\nsubseteq C\).  Let \(\phi\colon V\to E\) be an 
  \(E\)-linear map such that \(\ker\phi=C\) and \(1\in\phi(W)\) (such a map 
  clearly exists). We denote  by \(\fieldext{W}{C}\) the subfield  of \(E\)  
  generated by \(\phi(W)\). 
\end{definition}
The notation just introduced is unambiguous since the following result holds. 
\begin{lemma}\label{lem:linearalgebra}
  The subfield  \(\fieldext{W}{C}\) contains \(F\) and does not depend on 
  the choice of \(\phi\). Furthermore, \(\fieldext{W}{C}=F\) if and only if 
  \(\dim_{F}(W/W\cap C)=1\). Finally, 
  \(\ind{\fieldext{W}{C}}{F}\ge\dim_{F}(W/W\cap C)\), and, under the 
  additional assumption that \(\dim_{F}(W/W\cap C)\) is finite,  equality 
  holds 
   if 
  and only if \(\phi(W)\) is a field. 
\end{lemma}

\begin{proof}
    If \(\psi\) is another map with the required properties, then there 
    exists a non-zero \(u\) in \(E\) such that \(\psi=u\phi\): by 
    \cref{lem:cosetgeneratesamefield}, \(\psi(W)\) and \(\phi(W)\) generate 
    the same subfield. Since \(\phi(W)\) is an \(F\)-subspace of \(E\) 
    containing \(1\), it contains \(F\) and so does \(\fieldext{W}{C}\). 
    Moreover, \(\ind{\fieldext{W}{C}}{F}\ge\dim_{F}\phi(W)=\dim_{F}(W/W\cap 
    C)\) and, under the additional requirement that \(\dim_{F}(W/W\cap C)\) 
    is finite, equality holds if and only if \(\phi(W)=\fieldext{W}{C}\), 
    that is \(\phi(W)\) is a field; in particular, \(\fieldext{W}{C}=F\) if 
    and only if \(\dim_{F}(W/W\cap C)=1\). 
\end{proof}

\section{Finitely generated \(F\)-subalgebras of \(M\)}\label{sec:main}
In the rest of this paper \(E\supseteq F\) will be fields and \(M\) will be a 
Lie \(E\)-algebra of maximal class. We are interested in the description of 
the structure of  the Lie \(F\)-algebra generated by an \(F\)-subspace 
\(L_{1}\) of \(M_{1}\). First note that if \(\dim_{E}EL_{1}\le 1\) then the 
elements of \(L_{1}\) commute  and the Lie \(F\)-algebra generated by 
\(L_{1}\) is \(L_{1}\) itself equipped with the trivial Lie product. Thus we 
will assume that \(EL_{1}=M_{1}\): in particular,  \(L_{1}\nsubseteq 
C_{i}=\cent{M_1}{M_i}\) for every \(i\ge2\). We associate to \(L_{1}\) a 
sequence of subfields \(\{F_{i}\}_{i\ge2}\) of \(E\) by setting 
\(F_{i}\coloneqq \fieldext{L_{1}}{C_{i}}\) as in \cref{def:linearalgebra}.  
If \(K\) is a subfield of \(E\) containing \(F\) and \(T_{1}\) denotes the 
\(K\)-subspace \(KL_{1}\), then \(ET_{1}=M_{1}\), so we may similarly 
associate to \(T_{1}\) a sequence of subfields \(K_{i}\coloneqq 
\fieldext{T_{1}}{C_{i}}\). It easily turns out that \(K_{i}=K(F_{i})\) for 
every \(i\ge2\). 
\begin{remark}
When \(\dim_{F}L_{1}=2\), a sequence of integers \(\{d_{i}\}_{i\ge2}\) 
associated to \(L_{1}\) is defined in \cite{ACG+2023} by setting 
\(d_{i}\coloneqq \dim_{F}(L_{1}\cap C_{i})\): since  \(L_{1}\nsubseteq 
C_{i}\), the possible values for \(d_{i}\) are just \(0\) and \(1\) and 
\(d_{i}=1\) if and only if \(F_{i}=F\) so the sequence of the \(F_{i}\)'s 
(possibly) carries more information than the sequence of the \(d_{i}\)'s. 
\end{remark}
The following result, whose proof is immediate, will be used repeatedly 
hereafter. 
\begin{lemma}\label{lem:adjoint}
  If \(x\) is an element of 
  \(M_{1}\setminus C_{i}\) for some \(i\ge 2\), then the adjoint map 
  \(l\mapsto [l,x]\) is an \(E\)-isomorphism between \(M_{i}\) and 
  \(M_{i+1}\).
\end{lemma}

\begin{lemma}\label{lem:precovering}
  Let  \(U\) be an additive
  subgroup of \(M_{i}\) for some \(i\ge2\) and \(x\) an element in
  \(M_{1}\setminus C_{i}\). Then \(\cond{U}{E}=\cond{[U,x]}{E}\).
\end{lemma}
\begin{proof}
  By \cref{rem:increasingring}, the inclusion \(\cond{U}{E}\subseteq 
  \cond{[U,x]}{E}\) holds. To prove the reverse inclusion, we must show that 
  \(\alpha u\in U\) for every \(\alpha\in \cond{[U,x]}{E}\) and \(u\in U\). 
  Since \(U\) is an additive group, we have \([U,x]=\{[u,x]\mid u\in U\}\): 
  thus \([\alpha u,x]=\alpha[u,x]=[u',x]\) for some \(u'\in U\). By 
  \cref{lem:adjoint}, \(\alpha u=u'\in U\). 
\end{proof}
\begin{lemma}\label{lem:covering}
  Let \(U\) be a finite-dimensional \(F\)-subspace of \(M_{i}\) for some 
  \(i\ge 2\). Then \(\dim_{F}[U,L_{1}]\ge\dim_{F}U\) and equality holds if 
  and only if \(\cond{U}{E}\) contains \(F_i\), in which case 
  \(\cond{U}{E}=\cond{[U,L_{1}]}{E}\). 
\end{lemma}
\begin{proof}
  Take \(x\in L_{1}\setminus C_{i}\): by \cref{lem:adjoint}, 
  \(\dim_{F}[U,L_{1}]\ge\dim_{F}[U,x]=\dim_{F}U\). Equality holds if and only 
  \([U,x]=[U,L_{1}]\), that is \([U,y]\subseteq[U,x]\) for every \(y\in 
  L_{1}\). Let \(\phi\colon M_{1}\to E\) be the \(E\)-linear map such that 
  \(\ker\phi=C_{i}\) and \(\phi(x)=1\). Then \(y-\phi(y)x\in C_{i}\), so that 
  \([U,y]=\phi(y)[U,x]\). Thus, \([U,y]\subseteq[U,x]\) for every \(y\in 
  L_{1}\) if and only if \(\phi(L_{1})\subseteq \cond{[U,x]}{E}\). By 
  \cref{lem:precovering}, \(\cond{[U,x]}{E}=\cond{U}{E}\), and this is a 
  field by \cref{lem:modulering}. Thus \(\phi(L_{1})\subseteq 
  \cond{[U,x]}{E}\) if and only if \(\cond{U}{E}\supseteq  F_{i}\). Finally, 
  if equality holds, then \([U,x]=[U,L_{1}]\) and so 
  \(\cond{[U,L_{1}]}{E}=\cond{[U,x]}{E}=\cond{U}{E}\). 
\end{proof}

\begin{proposition}\label{prop:nearlymaximal}
   Let \(K\) be a subfield of \(E\) containing 
  \(F_{i}\) for every \(i\ge2\). If \(T\coloneqq\bigoplus_{i\ge 1}
  T_{i}\) is the \(K\)-algebra generated by \(L_{1}\) then 
  \(\dim_{K}T_{i}=\dim_{K}T_{1}-1\) for every \(i\ge 2\). 
\end{proposition}

\begin{proof}
  Let \(i\ge2\). Since \(EL_{1}=M_{1}\), it follows that \(L_{1}\nsubseteq 
  C_{i}\) and, a fortiori, \(T_{1}\nsubseteq C_{i}\). We choose \(x_{i}\in 
  L_{1}\setminus C_{i}\) and the \(E\)-linear map \(\phi_{i}\colon M_{1}\to 
  E\) such that \(\ker\phi_{i}=C_{i}\) and \(\phi_{i}(x_{i})=1\): thus 
  \(F_{i}\) is generated by \(\phi_{i}(L_{1})\) and 
  \(K_{i}\) is generated by \(\phi_{i}(T_{1})\). Since 
  \(T_{1}=KL_{1}\), we have 
  \(K_{i}=K(F_{i})=K\): 
  \cref{lem:linearalgebra} implies that \(\dim_{K}(T_{1}/T_{1}\cap C_{i})=1\) 
  so that \(T_{1}=Kx_{i}+(T_{1}\cap C_{i})\). Thus 
  \(T_{i+1}=[T_{i},T_{1}]=[T_{i},x_{i}]\) and 
  \(T_{2}=[T_{1},T_{1}]=[T_{1}\cap C_{2},x_{2}]\) (note that the elements in 
  \(C_{2}\) commute each other). By \cref{lem:adjoint}, 
  \(\dim_{K}T_{i+1}=\dim_{K}T_{i}\) for \(i\ge 2\) and 
  \(\dim_{K}T_{2}=\dim_{K}(T_{1}\cap C_{2})\), therefore 
  \(\dim_{K}T_{2}=\dim_{K}T_{1}-\dim_{K}(T_{1}/T_{1}\cap 
  C_{2})=\dim_{K}T_{1}-1\). 
\end{proof}

\begin{definition}
    The \emph{\(2\)-step field} of the Lie algebra \(L\) is the field \( \K\)  generated by \(\{F_i\}_{i\ge 2} \).  
\end{definition}

\begin{proposition}\label{prop:expanding}
  Suppose that  \(t\coloneqq\ind{\K}{F}\) is finite. Let \(X_{0}\) be a finite-dimensional \(F\)-subspace of \(M_{i}\) for some \(i\ge2\). Define recursively 
  \(X_{j}\coloneqq[X_{j-1},L_{1}]\) for \(j\ge1\). Then 
  \(\dim_{\K}\K X_{j}=\dim_{\K}\K X_{0}\) for every \(j\ge0\) and there exists an integer \(l\), independent of \(X_0\) and \(i\), such that 
  \(X_{l+k}\) is a \(\K\)-vector space for every \(k\ge0\). 
\end{proposition}

\begin{proof}
  Since \(\K X_{j}=[\K X_{j-1},L_{1}]\) for every \(j\ge1\), \cref{lem:covering} yields the first claim.

  In order to prove the second claim, we first consider the case \(\dim_{F}X_{0}=1\), so that \(\dim_{F}\K X_{j}=\dim_{F}\K X_{0}=\ind{\K}{F}=t\) for every \(j\ge0\).
  
  Since \(t\) is finite, there exists an integer \(r\) such that \(\K\) is generated by the \(F_{u}\)'s with \(u\le r\). Let \(l\coloneqq (t-1)r\). We need to prove that \(\cond{X_{l+k}}{E}\supseteq \K\) for every \(k\ge0\). We proceed by contradiction assuming, in virtue of \cref{rem:increasingring}, that \(\cond{X_{j}}{E}\nsupseteq F_{u}\) for some \(u\le r\) and every \(0\le j\le l\). By \cref{lem:specialisation}, we may 
  choose \(t-1\) indices \(0\le c_{1}<c_{2}<\dots<c_{t-1}<l\) such that 
  \(F_{i+c_{1}}=\dots=F_{i+c_{t-1}}=F_{u}\). Thus, \cref{lem:covering} 
  yields 
  \begin{multline*}
  1=\dim_{F}X_{0}\le\dim_{F}X_{c_{1}}<\dim_{F}X_{c_{1}+1} \le\dim_{F}X_{c_{2}}<\dim_{F}X_{c_{2}+1}\le\dots\le\\
  \dim_{F}X_{c_{t-1}}<\dim_{F}X_{c_{t-1}+1}\le\dim_{F}X_{l}.
  \end{multline*}
  Hence \(\dim_{F}X_{l}\ge t=\dim_{F}\K X_{l}\) and therefore 
  \(X_{l}=\K X_{l}\), a 
  contradiction. 

When \(\dim_{F}X_{0}>1\), decomposing \(X_{0}\) as sum of one-dimensional 
subspaces  completes the proof. 
\end{proof}

We are now in position to state a dichotomy for the \(F\)-subalgebras of a 
Lie \(E\)-algebra of maximal class, showing that they are either ideally 
\(r\)-constrained for some \(r\) or they are not just infinite. 

\begin{theorem}\label{thm:dichotomy}
  Let \(E\supseteq F\) be fields, \(M\) a Lie
  \(E\)-algebra of maximal class and \(L=\bigoplus_{i\ge1} L_{i}\) the Lie \(F\)-algebra generated by a finite-dimensional
 \(F\)-subspace \(L_{1}\) of \(M_{1}\)
  such that \(EL_{1}=M_{1}\). 
  Assume that the \(2\)-step field \(\K\) of \(L\) is a finite extension of \(F\). One of the following holds
  \begin{enumerate}
    \item  \(\dim_{\K}\K L_{1}=2\) and \(L\) is ideally \(r\)-constrained 
        for some \(r\); 
    \item  \(\dim_{\K}\K L_{1}>2\) and \(L\) is not just infinite. 
  \end{enumerate}
\end{theorem}
\begin{proof}
  Let \(T=\bigoplus_{i\ge1} T_{i}\) be the \(\K\)-algebra generated by \(L_{1}\), so
  \(T_{1}=\K L_{1}\). By \cref{prop:nearlymaximal},
  \(\dim_{\K}T_{i}=\dim_{\K}T_{1}-1\) for every \(i\ge 2\).

  Assume first that \(\dim_{\K}T_{1}=2\). Let \(I\) be a non-zero graded ideal
  of \(L\): set \(I_{j}\coloneqq I\cap L_{j}\) for every \(j\ge 1\). Let
  \(i\) be the smallest integer such that \(I_{i}\ne\{0\}\). Suppose that
  \(i\ge 2\): by \cref{prop:expanding}, there exists \(l\) independent of
  \(i\) such that \(I_{i+l}\) contains a non-zero \(\K\)-subspace of
  \(T_{i+l}\). Since \(\dim_{\K}T_{i+l}=1\) this implies that
  \(I_{i+l}=T_{i+l}\) and, a fortiori, \(I_{i+l}=L_{i+l}\). If \(i=1\), then
  \(I_{2}\ne\{0\}\) and the same argument yields \(I_{2+l}=L_{2+l}\). In
  any case \(L^{i}\supseteq I\supseteq L^{i+l+1}\), that is \(L\) is ideally
  \((l+1)\)-constrained.
  
  Assume now that \(\dim_{\K}T_{1}>2\) so that \(\dim_{\K}T_{i}>1\) for every 
  \(i\ge 2\). By \cref{prop:expanding}, \(L_{i}\) is a \(\K\)-vector space for 
  \(i\) large enough. Since \(\K L_{i}=T_{i}\), this means that \(L_{i}=T_{i}\) 
  for \(i\) large enough: in particular, \(\dim_{F}L_{i}\ge\dim_{\K}T_{i}>1\). 
  Choose such an \(i\) and let \(X\) be a non-zero proper \(\K\)-subspace 
  of \(L_{i}\). Denote by \(I\) and \(J\) the ideals generated by \(X\) 
  respectively in \(L\) and \(T\): clearly \(I\subseteq J\). 
  \cref{lem:covering} implies that \(\dim_{\K}(T_{j}\cap J)=\dim_{\K}X\) for 
  every \(j\ge i\) so that \(T_{j}\cap J\) is a proper subspace of \(T_{j}\): 
  as an obvious consequence \(L_{j}\cap I\) is a proper subspace of \(L_{j}\) 
  (remind that \(L_{j}=T_{j}\)): therefore \(I\) has infinite codimension in 
  \(L\), which is  not just infinite. 
\end{proof}

\begin{remark}
  There are assumptions that imply the finiteness of \(\ind{\K}{F}\). The simplest one is that \(\ind{E}{F}\) is finite. Another possibility is that \(E\) is an algebraic extension of \(F\) and \(M\) has finitely many distinct \(2\)-step centralisers.
\end{remark}

\begin{corollary}
  In the same hypotheses of \cref{thm:dichotomy}, if \(L\) is
  ideally constrained then \(\dim_{F}L_{i}=\ind{\K}{F}\) for every \(i\ge3\) and
  \(F_{i}=\K\) for every \(i\ge 2\).
\end{corollary}
\begin{proof}
  Let \(z\) be a (non-zero) homogeneous element of \(L\) of degree \(i\) with 
  \(i\ge2\): since \([z,L_{1}]=L_{i+1}\) we have 
  \(\dim_{F}L_{i+1}=\dim_{F}(L_{1}/L_{1}\cap C_{i})\). By 
  \cref{lem:specialisation}, 
  there are infinitely many values of \(i\) such that 
  \(\dim_{F}L_{i+1}=\dim_{F}L_{3}\). By \cref{lem:covering}, 
  \(\{\dim_{F}L_{i}\}_{i\ge 2}\) is a weakly increasing sequence, so it is 
  constant for \(i\ge 3\). By the proof of \cref{thm:dichotomy}, \(L_{i}\) is 
  a \(\K\)-vector space of dimension \(1\) for \(i\) large enough: thus 
  \(\dim_{F}L_{i}=\ind{\K}{F}\) for \(i\ge3\). By \cref{lem:linearalgebra}, 
  \(\ind{\K}{F}\ge\ind{F_{i}}{F}\ge\dim_{F}(L_{1}/L_{1}\cap 
  C_{i})=\dim_{F}L_{i+1}\) for \(i\ge2\), whence the last claim. 
\end{proof}
\begin{remark}
  If \(\dim_{F}L_{1}=2\) then, since
  \(\dim_{F}L_{1}\ge\dim_{\K}\K L_{1}\ge\dim_{E}EL_{1}=2\), the algebra
  generated by \(L_{1}\) is ideally \(r\)-constrained. However this can
  happen even if \(\dim_{F}L_{1}>2\): see \cref{ex:rconstrained}.
\end{remark}

\section{Examples and open problems}\label{sec:examples}

We now show that both cases of \cref{thm:dichotomy} actually occur. We remind 
that all the \(2\)-step centralisers of the metabelian Lie algebra of maximal 
class coincide. 

\begin{example}\label{ex:notjustinfinite}
  Let \(E\supsetneq F\) be fields with \(\ind{E}{F}\) finite and let \(M\) be 
  the metabelian Lie \(E\)-algebra of maximal class. Let \(L_{1}\) be the 
  \(F\)-subspace of \(M_{1}\) generated by an element \(x\) in 
  \(M_{1}\setminus C_{2}\) and an \(F\)-subspace \(U\) of \(C_{2}\) of 
  dimension \(d\ge 2\). Take the \(E\)-linear map \(\phi\colon M_{1}\to E\)
  such that \(\phi(x)=1\) and \(\ker\phi=C_{2}\); thus 
  \(\K=F_{2}=F\) and \(\dim_{\K}\K L_{1}=\dim_{F}L_{1}\ge 3\) 
  so that the Lie \(F\)-algebra generated by \(L_{1}\) is not just infinite. 
\end{example}

\begin{example}\label{ex:rconstrained}
  Let \(\alpha\) be an algebraic element of degree \(d\ge 2\) over a field \(F\) 
  and let \(E\coloneqq F(\alpha)\). Let \(M\) be a Lie algebra of maximal 
  class over \(E\) with at most two distinct \(2\)-step centralisers. Denote 
  by \(C\) a \(2\)-step centraliser and choose a non-zero element \(y\) in 
  \(C\). Choose another non-zero element \(x\) of \(M_{1}\) as follows: if 
  \(M\) has two distinct \(2\)-step centralisers, \(x\) belongs to the 
  \(2\)-step centraliser different from \(C\); otherwise \(x\) is just an 
  element not in \(C\). Let \(L_{1}\) be the \(F\)-subspace of \(M_{1}\) 
  generated by \(x\), \(\alpha x\) and \(y\). Given an index \(i\) such that 
  \(C_{i}=C\), consider the \(E\)-linear map \(\phi\colon M_{1}\to E\) such 
  that \(\phi(x)=1\) and \(\ker\phi=C\); thus 
  \(F_{i}=F(\alpha)=E=\K\) and \(\dim_{\K}\K L_{1}=2\) so that 
  the Lie \(F\)-algebra \(L\) generated by \(L_{1}\) is ideally 
  \(r\)-constrained for some \(r\).

  We now compute \(r\): some notation is needed. 
  Given a non-zero homogeneous element \(z\) of \(L\) of degree \(i\), we 
  denote by \(\covdeg{z}\) the smallest \(k\) such that 
  \([z,\prescript{}{k}L_{1}]=L_{i+k}\). If \(r\) is the maximum of 
  \(\covdeg{z}\) as \(z\) ranges over the set of homogeneous elements of 
  \(L\), then \(L\) is ideally \(r\)-constrained but not ideally 
  \((r-1)\)-constrained. 
  
  Given a non-zero homogeneous element \(z\) of degree \(i\ge 2\) of \(M\) and 
  a nonnegative integer \(t\), we denote by \(U(z,t)\) the set of the 
  elements of the form \(p(\alpha)z\) as \(p(\alpha)\) ranges over the set of 
  the polynomials in \(\alpha\) with coefficients in \(F\) and degree at most 
  \(t\). The set \(U(z,t)\) is clearly an \(F\)-subspace of \(M_{i}\) and its 
  dimension is  \(t+1\) for \(t<d\) and \(d\) for \(t\ge d-1\). If 
  \(C_{i}=C\) then \([U(z,t),L_{1}]=U([z,x],t+1)\); this is the only 
  possibility if \(M\) is metabelian. If \(M\) is not metabelian there are 
  indices \(i\) such that \(C_{i}\ne C\): for those indices, we have 
  \([U(z,t),L_{1}]=U([z,y],t)\). In particular, 
  \(\dim_{F}U(z,t)<\dim_{F}[U(z,t),L_{1}]\) if and only if \(C_{i}=C\) and 
  \(t<d-1\).

  If \(w\) is a non-zero homogeneous element of \(M\) of the same degree 
  \(i\ge 2\) as \(z\) and \(t\) and \(s\) are nonnegative integers such that 
  \(U(z,t)\) is strictly contained in \(U(w,s)\) but 
  \([U(z,t),L_{1}]=[U(w,s),L_{1}]\), then \(C_{i}=C\) and \(t=d-2\). 
  
  Given an integer \(i\ge 2\) and a nonnegative integer \(k\), we denote by 
  \(m_{i,k}\) the number of integers \(j\) such that \(i\le j<i+k\) and 
  \(C_{j}=C\); we denote by \(m_{i}\) the smallest \(k\) such that 
  \(m_{i,k}=d-1\). Since \(m_{i,k}\le k\), we have that \(m_{i}\ge d-1\), and 
  equality holds for every \(i\) when \(M\) is metabelian. An easy induction 
  shows that \([U(z,t),\prescript{}{k}L_{1}]=U(w,t+m_{i,k})\) for some 
  homogeneous element \(w\) of degree \(i+k\). In particular, since 
  \(L_{2}=U([y,x],1)\), we have that for every \(j\ge2\), 
  \(L_{j}=U(w,1+m_{2,j-2})\) for some homogeneous element \(w\) of degree 
  \(j\). 
  
  If \(z\) is an element of \(L\) of degree \(i\ge2\), then the 
  \(F\)-subspace generated by \(z\), that is \(U(z,0)\), is strictly 
  contained in \(L_{i}=U(w,t)\) for some \(w\) of degree \(i\) and some 
  positive integer \(t\). By the previous discussion, 
  \([U(z,0),\prescript{}{k}L_{1}]=[L_{i},\prescript{}{k}L_{1}]\) if and only 
  if \(k\ge m_{i}\), so that \(\covdeg{z}=m_{i}\). It remains to consider 
  non-zero elements in \(L_{1}\): if \(z\) is such an element then 
  \([z,L_{1}]\) is a non-zero subspace of \(L_{2}\) and therefore 
  \(\covdeg{z}\le m_{2}+1\); if we choose \(z\coloneqq x\) then \([x,L_{1}]\) 
  is the \(F\)-vector space generated by \([x,y]\) and therefore 
  \(\covdeg{z}=m_{2}+1\). Summarising, the maximum of \(\covdeg{z}\), that is 
  the integer \(r\) such that \(L\) is ideally \(r\)-constrained but not 
  \((r-1)\)-constrained, is \(m_{2}+1\) if \(m_{2}=\max\{m_{i}\}\) and is 
  \(\max\{m_{i}\}\) otherwise. In particular, if \(M\) is the metabelian Lie 
  algebra of  maximal class, then \(r=d\) thus showing that there exist 
  examples of ideally \(r\)-constrained Lie algebras for every positive 
  integer \(r\).  
 \end{example}  

We close the paper with some problems left open to the reader. 

\begin{problem}
    In \cref{thm:dichotomy} we have assumed that \(\K\supseteq F\) is a finite 
    extension. What can be said when the degree of the extension is infinite?  
    For instance, let \(E=F(\alpha)\), where \(\alpha\) is transcendental 
    and \(M\) be the metabelian Lie \(E\)-algebra of maximal class.  Consider  
    \(L\) to be the  \(F\)-algebra generated by \(x\) and \(\alpha x+y\), 
    where \(x\in M_1\setminus C_2\) and \(0\ne y\in C_2\). Then \(E=\K\) and an argument 
    similar to \cref{ex:rconstrained} shows that \(L\) is the free 
    \(2\)-generated metabelian algebra. Note also that \(L\) coincides with the algebra 
    constructed in~\cite[Lemma~1]{GMY2001}.
\end{problem}
\begin{problem}
Is it possible to compute the smallest \(r\) in the first case of 
\cref{thm:dichotomy}? 
\end{problem}
\begin{problem}
  What can be said if we take subalgebras of a thin algebra or even more
  generally of an ideally \(r\)-constrained algebra over \(E\)? 
\end{problem}

\printbibliography
\end{document}